\documentclass[12pt, a4paper]{article}
\usepackage{amsfonts}
\usepackage{mathrsfs}
\usepackage{latexsym}
\usepackage{xy}
\usepackage{amsfonts,amsmath,amssymb,amsthm}
\usepackage{color}
\xyoption{all}

\newcommand{\bcen}{\begin{center}}     \newcommand{\ecen}{\end{center}}
\newcommand{\bay}{\begin{array}}      \newcommand{\eay}{\end{array}}
\newcommand{\beq}{\begin{eqnarray*}}      \newcommand{\eeq}{\end{eqnarray*}}

\def\dim{\mathrm{dim}}
\def\End{\mathrm{End}}

\def\Ext{\mathrm{Ext}}

\def\gl{\mathrm{gl.dim}}

\def\Hom{\mathrm{Hom}}

\def\mod{\mathrm{mod}}
\def\Mod{\mathrm{Mod}}

\def\pd{\mathrm{pd}}

\def\proj{\mathrm{proj}}

\def\rad{\mathrm{rad}}

\def\RHom{\mathrm{RHom}}

\def\thick{\mathrm{thick}}

\begin{document}

\newtheorem{theorem}{Theorem}
\newtheorem{proposition}{Proposition}
\newtheorem{lemma}{Lemma}
\newtheorem{corollary}{Corollary}
\newtheorem{remark}{Remark}
\newtheorem{example}{Example}
\newtheorem{definition}{Definition}
\newtheorem*{conjecture}{Conjecture}
\newtheorem{question}{Question}

\title{\large\bf Jordan-H\"{o}lder theorems for derived categories
of derived discrete algebras}

\author{\large Yongyun Qin}

\date{\footnotesize
Chinese Academy of Sciences, Beijing 100190, P.R. China.\\ E-mail:
qinyongyun2006@126.com}

\maketitle

\begin{abstract}
For any positive integer $n$, $n$-derived-simple derived discrete
algebras are classified up to derived equivalence. Furthermore, the
Jordan-H\"{o}lder theorems for all kinds of derived categories of
derived discrete algebras are obtained.
\end{abstract}

\medskip


\medskip


\section{\large Introduction}

\indent\indent Throughout this paper, $k$ is an algebraically closed
field and all algebras are finite dimensional associative
$k$-algebras with identity. Recollements of triangulated categories
were introduced by Beilinson, Bernstein and Deligne \cite{BBD82},
and play an important role in algebraic geometry and representation
theory. Recollements, more general $n$-recollements \cite{HQ14}, of
derived categories of algebras provide a reduction technique for
some homological properties such as the finiteness of global
dimension \cite{Wie91,Koe91,AKLY13}, the finiteness of finitistic
dimension \cite{Hap93,CX14}, the finiteness of Hochschild dimension
\cite{Han14} and homological smoothness \cite{HQ14}, and some
homological invariants such as $K$-theory
\cite{TT90,Yao92,Nee04,Sch06,CX12,AKLY13}, Hochschild homology and
cyclic homology \cite{Kel98} and Hochschild cohomology \cite{Han14}.
Moreover, some homological conjectures such as the finitistic
dimension conjecture \cite{Hap93,CX14}, the Hochschild homology
dimension conjecture \cite{HQ14}, the Cartan determinant conjecture
\cite{HQ14}, the Gorenstein symmetry conjecture \cite{HQ14} and the
(dual) vanishing conjecture \cite{Y14}, can be reduced to
$n$-derived-simple algebras by $n$-recollements for appropriate $n
\geq 1$. A recollement of derived categories of algebras is a short
exact sequence of derived categories of algebras, and thus leads to
the concepts of composition series and composition factors and
Jordan-H\"{o}lder theorem analogous to those in group theory and
module theory. Indeed, the composition factors in this context are
the derived categories of derived simple algebras which were studied
in \cite{Wie91,Happ91,LY12,LY13,AKLY13}. Meanwhile, the
Jordan-H\"{o}lder theorem of derived category was established for
hereditary algebras \cite{AKLO12} and later for piecewise hereditary
algebras \cite{AKL12}. Nonetheless, this theorem was disproved for
certain infinite dimensional algebras \cite{CX11,CX12}, and later
for a finite dimensional algebra \cite{AKLY13}. Hence, the
Jordan-H\"{o}lder theorem of the derived category of an algebra is a
little bit subtle.

Derived discrete algebras were introduced by Vossieck \cite{Voss01},
which are crucial in the Brauer-Thrall type theorems for derived
category \cite{HZ13} and were explored in
\cite{BGS04,Bobi11,BPP13,BPP14,HanZ13,BK14}. In particular, Vossieck
proved that a basic connected derived discrete algebra is isomorphic
to either a piecewise hereditary algebra of Dynkin type or some
one-cycle gentle algebra which does not satisfying the
clock-condition \cite{Voss01}. The latter was further classified up
to derived equivalence by Bobi\'{n}ski, Gei{\ss} and Skowro\'{n}ski
\cite{BGS04}. Synthesizing the relevant results of Vossieck
\cite{Voss01}, Happel \cite{Hap88} and Bekkert-Merklen \cite{BM03},
all indecomposable objects in the bounded derived category of a
derived discrete algebra are fairly clear.

In this paper, we will classify all $n$-derived-simple derived
discrete algebras up to derived equivalence, and prove the
Jordan-H\"{o}lder theorems for all kinds of derived categories of
derived discrete algebras. More precisely, we will show that, for
any positive integer $n$, a derived discrete algebra is
$n$-derived-simple if and only if it is derived equivalent to the
underlying field $k$ or a 2-truncated cycle algebra. In order to
obtain simultaneously the Jordan-H\"{o}lder theorems for unbounded,
upper bounded, lower bounded, bounded derived categories of derived
discrete algebras, we will introduce $n$-composition-series and
$n$-composition-factors of the unbounded derived category of an
algebra. Furthermore, we will prove that the unbounded derived
category of any derived discrete algebra admits an
$n$-composition-series whose $n$-composition-factors are independent
on the choice of $n$-composition-series up to reordering and derived
equivalence. In particular, the Jordan-H\"{o}lder theorem holds for
the unbounded (resp. upper bounded, lower bounded, bounded) derived
category of any derived discrete algebra.

This paper is organized as follows: In section 2, we will study the
$n$-recollements of derived categories of derived discrete algebras.
In Section 3, we will classify all $n$-derived-simple derived
discrete algebras up to derived equivalence. In section 4, we will
define the $n$-composition-series and $n$-composition-factors of the
unbounded derived category of an algebra, and prove the
Jordan-H\"{o}lder theorems for derived categories of derived
discrete algebras.

\section{\large $n$-recollements on derived discrete algebras}

\indent\indent In this section, we will study the $n$-recollements
of derived categories of derived discrete algebras.

Let $A$ be an algebra. Denote by $\Mod A$ the category of right
$A$-modules, and by $\mod A$ (resp. $\proj A$) its full subcategory
consisting of all finitely generated modules (resp. finitely
generated projective modules). Denote by $\mathcal{D}(\Mod A)$ or
just $\mathcal{D}A$ the unbounded derived category of $A$. Denote by
$\mathcal{D}^b(\mod A)$ or just $\mathcal{D}^b(A)$ the bounded
derived category of $A$. Denote by $K^b(\proj A)$ the homotopy
category of cochain complexes of objects in $\proj A$. An object $X$
in $\mathcal{D}A$ is said to be {\it compact} if
$\Hom_{\mathcal{D}A}(X,-)$ commutes with direct sums. It is
well-known that up to isomorphism the objects in $K^{b}(\proj A)$
are precisely all the compact objects in $\mathcal{D}A$ (Ref.
\cite{BN}).

An algebra $A$ is said to be {\it derived discrete} provided for
every vector $\mathbf{d} = (d_p)_{p \in \mathbb{Z}} \in
\mathbb{N}^{(\mathbb{Z})}$ there are only finitely many isomorphism
classes of indecomposable objects $X$ in $\mathcal{D}^b(A)$ of
cohomology dimension vector $(\dim_kH^p(X))_{p \in \mathbb{Z}}
\linebreak = \mathbf{d}$ (cf. \cite{Voss01}). Some characterizations
of derived discrete algebras in terms of (global) cohomological
range and (global) cohomological length are provided in \cite{HZ13}.

Vossieck proved in \cite{Voss01} that a basic connected derived
discrete algebra is isomorphic to either a piecewise hereditary
algebra of Dynkin type or a one-cycle gentle algebra not satisfying
the clock-condition. The latter was shown by
Bobi\'{n}ski-Gei{\ss}-Skowro\'{n}ski in \cite{BGS04} to be derived
equivalent to a gentle algebra $\Lambda(r,s,t)$ with $1 \leq r \leq
s$ and $t \geq 0$, which is the bound quiver algebra \cite{ASS06}
given by the quiver $Q(s,t) :=$
$$\xymatrix{ &&&& 1 \ar[r]^-{\alpha_{1}}
& \cdots \ar[r]^-{\alpha_{s-r-2}} &  s-r-1 \ar[dr]^-{\alpha_{s-r-1}}
& \\ (-t) \ar[r]^-{\alpha_{-t}} & \cdots \ar[r]^-{\alpha_{-2}} &
(-1) \ar[r]^-{\alpha_{-1}} & 0 \ar[ur]^-{\alpha_0} &&&& s-r
\ar[dl]^-{\alpha_{s-r}} \\ &&&& s-1 \ar[ul]^-{\alpha_{s-1}} & \cdots
\ar[l]^-{\alpha_{s-2}} & s-r+1 \ar[l]^-{\alpha_{s-r+1}} & }$$ with
the relations $\alpha_{s-1}\alpha_0, \alpha_{s-2}\alpha_{s-1},
\cdots , \alpha_{s-r}\alpha_{s-r+1}$. Thus, up to derived
equivalence, a derived discrete algebra can be decomposed as
$\oplus^m_{p=1}A_p$ where $m \geq 1$ and $A_p =
\Lambda(s_p,s_p,t_p)$ for $s_p \geq 1$, $t_p \geq 0$, $1 \leq p \leq
u$ and $0 \leq u \leq m$; $A_p = \Lambda(r_p,s_p,t_p)$ for $s_p >
r_p \geq 1$, $t_p \geq 0$, $u+1 \leq p \leq v$ and $u \leq v \leq
m$; $A_p$ is a basic connected piecewise hereditary algebra of
Dynkin type for $v+1 \leq p \leq m$. Obviously, the algebra
$\Lambda(r,s,t)$ is of finite global dimension if and only if $s >
r$. The algebra $\Lambda(s,s,0)$ with $s \geq 1$ is called a {\it
2-truncated cycle algebra}.

Now we observe the noncompact indecomposable objects in
$\mathcal{D}^b(\Lambda (s,s,t))$. For $0 \leq p \leq s-1 $, let
$X_p$ be the simple $\Lambda(s,s,t)$-module $S_p$ corresponding to
the vertex $p$ of the quiver $Q(s,t)$. For $-t \leq q \leq -1$, let
$Y_q$ be the indecomposable $\Lambda(s,s,t)$-module $P_q/S_1$ if $s
\geq 2$ and $P_q/S_0$ if $s=1$, where $P_q$ is the indecomposable
projective $\Lambda(s,s,t)$-module corresponding to the vertex $q$
of the quiver $Q(s,t)$. Note that $Y_q$ has a minimal projective
resolution
$$\xymatrix{\cdots \cdots \ar[r]^-{\alpha_0} & P_0
\ar[r]^-{\alpha_{s-1}} & P_{s-1} \ar[r]^-{\alpha_{s-2}} & \cdots
\ar[r]^-{\alpha_2} & P_2 \ar[r]^-{\alpha_1} & P_1 \ar[rr]^-{\alpha_q
\cdots \alpha_{-1}\alpha _0} && P_q \ar[r] & 0}$$ where all
differentials are given by the left multiplications of the paths.

\begin{lemma} \label{Lemma-Hom-RHom}
Let $A = \Lambda (s,s,t)$ with $s \geq 1 $ and $t \geq 0$. Then, up
to isomorphism and shift, $X_0, X_1, \cdots , X_{s-1}$ and $Y_{-t},
Y_{-t+1}, \cdots , Y_{-1}$ are all noncompact indecomposable objects
in $\mathcal{D}^b(A)$. Moreover, for all noncompact indecomposable
objects $X$ and $Y$ in $\mathcal{D}^b(A)$, the following statements
hold:

{\rm (1)} $\Hom_{\mathcal{D}A}(X,X[h]) \cong k$ if
$h=0,s,2s,\cdots$, and $=0$ otherwise.

{\rm (2)} $\RHom_A(X,Y) \ncong 0.$
\end{lemma}

\begin{proof}
For the gentle algebra $A = \Lambda (s,s,t)$, it is not difficult to
see $GSt_c = \{\alpha_p, (\alpha_{q} \cdots
\alpha_{-1} \alpha_0) \ | \ 0 \leq p \leq s-1, \ -t \leq q \leq -1\}$ and
$GSt^c = \{\alpha^{-1}_p, \linebreak (\alpha^{-1}_0\alpha^{-1}_{-1} \cdots
\alpha^{-1}_q) \ | \ 0 \leq p \leq s-1, \ -t \leq q \leq
-1\}$ (Ref. \cite[4.3]{BM03}). Since generalized strings are defined
up to inverse, i.e., the equivalence relation $\cong_s$ (Ref.
\cite[4.1]{BM03}), we have $GSt_c = GSt^c = \{\alpha_p, (\alpha_q
\cdots \alpha_{-1}\alpha_0) \ | \ 0 \leq p \leq s-1, \ -t \leq q
\leq -1\}$. Due to \cite[Definition 2, Lemma 6, and Theorem
3]{BM03}, up to isomorphism and shift, $X_0, X_1, \cdots , X_{s-1}$
and $Y_{-t}, Y_{-t+1}, \cdots , Y_{-1}$ are all noncompact
indecomposable objects in $\mathcal{D}^b(A)$.

Let $X$ and $Y$ be two noncompact indecomposable objects in
$\mathcal{D}^b(A)$.

(1) There are altogether two cases:

(1.1) $X = X_p, \ 0 \leq p \leq s-1$. Since $X$ is the simple
$A$-module $S_p$ corresponding to the vertex $p$ of the quiver
$Q(s,t)$, it is easy to calculate $\Hom_{\mathcal{D}A}(X,X[h]) \cong
k$ if $h=0,s,2s,\cdots$, and $=0$ otherwise.

(1.2) $X = Y_q, \ -t \leq q \leq -1$. In the following, we only
consider the case $s \geq 2$. For the case $s=1$, it is enough to
replace $S_1$ with $S_0$ in appropriate places. We have a short
exact sequence $S_1 \rightarrowtail P_q \twoheadrightarrow Y_q$,
which induces a triangle $P_q \rightarrow Y_q \rightarrow S_1[1]
\rightarrow$ in $\mathcal{D}A$. Applying the derived functor
$\RHom_A(Y_q,-)$ to this triangle, we obtain a triangle
$\RHom_A(Y_q,P_q) \rightarrow \RHom_A(Y_q,Y_q) \rightarrow
\RHom_A(Y_q,S_1[1]) \rightarrow$ in $\mathcal{D}k$. It is not
difficult to prove $\RHom_A(Y_q,P_q) \cong 0$. Thus we have
$\RHom_A(Y_q,Y_q) \cong \RHom_A(Y_q,S_1[1])$. Now (1.2) follows from
the obvious fact $H^h(\RHom_A(Y_q,S_1[1])) \cong k$ if
$h=0,s,2s,\cdots$, and $=0$ otherwise.

(2) There are altogether four cases:

(2.1) $(X=X_p)\wedge(Y=X_{p'}), \ 0 \leq p,p' \leq s-1$. Obviously,
there exist infinitely many $h \in \mathbb{N}$ such that $\Ext_{A}^h
(S_p,S_{p'}) \neq 0$. Thus $\RHom_A(X_p,X_{p'}) \ncong 0$.

(2.2) $(X=X_p)\wedge(Y=Y_q), \ 0 \leq p \leq s-1, \ -t \leq q \leq
-1$. Applying the derived functor $\RHom_A(X_p,-)$ to the triangle
$P_q \rightarrow Y_q \rightarrow S_1[1] \rightarrow$ in
$\mathcal{D}A$, we obtain a triangle $\RHom_A(X_p,P_q) \rightarrow
\RHom_A(X_p,Y_q) \rightarrow \RHom_A(X_p,S_1[1]) \rightarrow$ in
$\mathcal{D}k$ which gives a long exact sequence. We can show
$\Hom_{\mathcal{D}A}(X_p,P_q[h])=0$ for all $h \neq 0$. Thus
$\Hom_{\mathcal{D}A}(X_p,Y_q[h]) \cong
\Hom_{\mathcal{D}A}(X_p,S_1[h+1])$ for all $h \geq 1$. Since
$\Ext_A^h(S_p,S_1) \neq 0$ for infinitely many $h>0$, there exist
infinitely many $h \in \mathbb{N}$ such that
$\Hom_{\mathcal{D}A}(X_p,Y_q[h]) \neq 0$. Hence $\RHom_A(X_p,Y_q)
\ncong 0$.

(2.3) $(X=Y_q)\wedge(Y=X_p), \ -t \leq q \leq -1, \ 0 \leq p \leq
s-1$. Applying the derived functor $\RHom_A(-,S_p)$ to the triangle
$P_q \rightarrow Y_q \rightarrow S_1[1] \rightarrow$ in
$\mathcal{D}A$, we obtain a triangle $\RHom_A(S_1[1],S_p)
\rightarrow \RHom_A(Y_q,S_p) \rightarrow \RHom_A(P_q,S_p)
\rightarrow$ in $\mathcal{D}k$. Obviously, $\RHom_A(P_q,S_p) \cong
0$. Therefore, we have $\RHom_A(Y_q,S_p) \cong \RHom_A(S_1[1],S_p)
\ncong 0$ by (2.1), i.e., $\RHom_A(Y_q,X_p) \ncong 0.$

(2.4) $(X=Y_q)\wedge(Y=Y_{q'}), \ -t \leq q,q' \leq -1$. Applying
the derived functor $\RHom_A(Y_q,-)$ to the triangle $P_{q'}
\rightarrow Y_{q'} \rightarrow S_1[1] \rightarrow$ in
$\mathcal{D}A$, we obtain a triangle $\RHom_A(Y_q,P_{q'})
\rightarrow \RHom_A(Y_q,Y_{q'}) \rightarrow \RHom_A(Y_q,S_1[1])
\rightarrow$ in $\mathcal{D}k$. It is not difficult to show
$\Hom_{\mathcal{D}A}(Y_q,P_{q'}[h])=0$ for all $h \neq 1$. Thus
$\Hom_{\mathcal{D}A}(Y_q,Y_{q'}[h]) \cong
\Hom_{\mathcal{D}A}(Y_q,S_1[h+1])$ for all $h \geq 2$. Note that
$\Hom_{\mathcal{D}A}(Y_q,S_1[h+1]) \neq 0$ for infinitely many $h
\in \mathbb{N}$. Hence, there exist infinitely many $h \in
\mathbb{N}$ such that $\Hom_{\mathcal{D}A}(Y_q,Y_{q'}[h]) \neq 0$.
Therefore, $\RHom_A(Y_q,Y_{q'}) \ncong 0.$
\end{proof}

Before giving some results on the $n$-recollements of derived
categories of derived discrete algebras, we recall the definition of
$n$-recollement of triangulated categories.

Let $\mathcal{T}_1$, $\mathcal{T}$ and $\mathcal{T}_2$ be
triangulated categories. A {\it recollement} of $\mathcal{T}$
relative to $\mathcal{T}_1$ and $\mathcal{T}_2$ is given by
$$\xymatrix@!=4pc{ \mathcal{T}_1 \ar[r]^{i_*=i_!} & \mathcal{T} \ar@<-3ex>[l]_{i^*}
\ar@<+3ex>[l]_{i^!} \ar[r]^{j^!=j^*} & \mathcal{T}_2
\ar@<-3ex>[l]_{j_!} \ar@<+3ex>[l]_{j_*}}$$ such that

(R1) $(i^*,i_*), (i_!,i^!), (j_!,j^!)$ and $(j^*,j_*)$ are adjoint
pairs of triangle functors;

(R2) $i_*$, $j_!$ and $j_*$ are full embeddings;

(R3) $j^!i_*=0$ (and thus also $i^!j_*=0$ and $i^*j_!=0$);

(R4) for each $X \in \mathcal {T}$, there are triangles

$$\begin{array}{l} j_!j^!X \rightarrow X  \rightarrow i_*i^*X  \rightarrow
\\ i_!i^!X \rightarrow X  \rightarrow j_*j^*X  \rightarrow
\end{array}$$ where the arrows to and from $X$ are the counits and the
units of the adjoint pairs respectively \cite{BBD82}.

Let $\mathcal{T}_1$, $\mathcal{T}$ and $\mathcal{T}_2$ be
triangulated categories, and $n$ a positive integer. An {\it
$n$-recollement} of $\mathcal{T}$ relative to $\mathcal{T}_1$ and
$\mathcal{T}_2$ is given by $n+2$ layers of triangle functors
$$\xymatrix@!=4pc{ \mathcal{T}_1 \ar@<+1ex>[r] \ar@<-3ex>[r]_\vdots & \mathcal{T}
\ar@<+1ex>[r]\ar@<-3ex>[r]_\vdots \ar@<-3ex>[l] \ar@<+1ex>[l] &
\mathcal{T}_2 \ar@<-3ex>[l] \ar@<+1ex>[l]}$$ such that every
consecutive three layers form a recollement \cite{HQ14}.

From now on, we only focus on the $n$-recollements of derived
categories of algebras, i.e., all three triangulated categories in
an $n$-recollement are the derived categories of algebras.

\begin{proposition} \label{Proposition-n-recollement}
Let $A$ be a derived discrete algebra. Then every recollement of
$\mathcal{D}A$ relative to $\mathcal{D}B$ and $\mathcal{D}C$ can be
extended to an $n$-recollement for all $n \geq 1$.
\end{proposition}

\begin{proof}
If gl.dim$A$ $<\infty$ then this is clear by \cite[Proposition
3.3]{AKLY13}. If gl.dim$A$ $=\infty$, all the connected directed
summands of $A$ that are of infinite global dimension are derived
equivalent to $\Lambda(s,s,t)$ for some $s \geq 1$ and $t \geq 0$.
Assume that there is a recollement
$$\xymatrix@!=4pc{ \mathcal{D}B \ar[r]^{i_*}
& \mathcal{D}A \ar@<-3ex>[l]_{i^*} \ar@<+3ex>[l]_{i^!} \ar[r]^{j^!}
& \mathcal{D}C \ar@<-3ex>[l]_{j_!} \ar@<+3ex>[l]_{j_*}}.$$ Then
$i_*B \in \mathcal{D}^b(A)$ by \cite[Lemma 2]{HQ14}. If $i_*B$ is
not compact then there exists an indecomposable projective
$B$-module $P$ such that $i_*P$ is not compact. Therefore, $i_*P$ is
a noncompact indecomposable object in
$\mathcal{D}^b(\Lambda(s,s,t))$ for some $s \geq 1$ and $t \geq 0$.
By Lemma~\ref{Lemma-Hom-RHom}, $i_*P$ is isomorphic to a shift of
$X_p$ or $Y_q$, and $\Hom_{\mathcal{D}A}(i_*P,i_*P[h]) \neq 0$ for
infinitely many $h$. It is a contradiction, since $i_*B$, and thus
$i_*P$, is exceptional. Hence, $i_*B$ is compact and the recollement
can be extended to a $2$-recollement by \cite[Lemma 3]{HQ14}.
Furthermore, the recollement can be extended to an $n$-recollement
for all $n \geq 1$ by induction.
\end{proof}

The following result is a generalization of
\cite[Proposition]{Voss01}, which states that derived discrete
algebras are invariant under derived equivalence.

\begin{proposition} \label{Proposition-DerivedDiscrete}
Let $A$, $B$ and $C$ be algebras, and $\mathcal{D}A$ admit a
recollement relative to $\mathcal{D}B$ and $\mathcal{D}C$. If $A$ is
derived discrete then so are $B$ and $C$.
\end{proposition}

\begin{proof}
According to Proposition~\ref{Proposition-n-recollement} and
\cite[Lemma 2]{HQ14}, we have a recollement
$$\xymatrix@!=4pc{ \mathcal{D}B \ar[r]^{i_*}
& \mathcal{D}A \ar@<-3ex>[l]_{i^*} \ar@<+3ex>[l]_{i^!} \ar[r]^{j^!}
& \mathcal{D}C \ar@<-3ex>[l]_{j_!} \ar@<+3ex>[l]_{j_*}},$$ where all
six triangle functors can be restricted to $\mathcal{D}^b(\mod)$.
Since $i_*$ is fully faithful, it induces a map from the set of
isomorphism classes of indecomposable objects in $\mathcal{D}^b(B)$
to the set of isomorphism classes of indecomposable objects in
$\mathcal{D}^b(A)$. Due to $i^*i_* \cong 1$, this map is injective.
Moreover, by the same estimate on the cohomological dimension vector
as \cite[Section 1.1]{Voss01}, we can show that the cohomological
dimension vectors of the images of indecomposable objects in
$\mathcal{D}^b(B)$ of a fixed cohomological dimension vector under
the triangle functor $i_*$ are ``bounded''. Thus $B$ is derived
discrete. Similarly, we can prove $C$ is derived discrete by
considering the triangle functor $j_*$ instead of $i_*$.
\end{proof}

\begin{remark}{\rm The converse of Proposition~\ref{Proposition-DerivedDiscrete}
is not true in general. Kronecker algebra $A$, which is nothing but
the path algebra $kQ$ of the quiver $Q$ given by two vertices 1 and
2 and two arrows from 1 to 2, provides a counterexample. Indeed,
though $\mathcal{D}A$ admits a recollement relative to
$\mathcal{D}k$ and $\mathcal{D}k$, $A$ is not derived discrete since
all derived discrete algebras are of finite representation type.
}\end{remark}

\section{\large $n$-derived-simple derived discrete algebras}

\indent\indent In this section, we will classify all
$n$-derived-simple derived discrete algebras up to derived
equivalence.

Let $n$ be a positive integer. An algebra $A$ is said to be {\it
$n$-derived-simple} if its derived category $\mathcal{D}A$ admits no
nontrivial $n$-recollements of derived categories of algebras
\cite{HQ14}. Note that 1- (resp. 2-, 3-) derived-simple algebras
here are just the $\mathcal{D}(\Mod)$- (resp. $\mathcal{D}^-(\Mod)$-
or $\mathcal{D}^+(\Mod)$-, $\mathcal{D}^b(\mod)$- or
$\mathcal{D}^b(\Mod)$-) derived-simple algebras in \cite{AKLY13}.

Now we focus on the algebra $\Lambda(s,s,t)$, where $s \geq 1$ and $t \geq 0$.
The following lemma shows that $\Lambda(s,s,t)$ is rigid
when deconstructing its derived category by recollements.

\begin{lemma} \label{Lemma-Recollement-B-C}
Let $A = \Lambda(s,s,t)$ with $s \geq 1$ and $t \geq 0$, and
$\mathcal{D}A$ admit a recollement relative to $\mathcal{D}B$ and
$\mathcal{D}C$. Then either $\gl C < \infty$ and $B$ is derived
equivalent to $\Lambda(s,s,t') \oplus B'$ with $0 \leq t' \leq t$
and $\gl B' < \infty$, or $\gl B < \infty$ and $C$ is derived
equivalent to $\Lambda(s,s,t') \oplus C'$ with $0 \leq t' \leq t$
and $\gl C' < \infty$.
\end{lemma}

\begin{proof}
If $\mathcal{D}A$ admits a recollement relative to $\mathcal{D}B$
and $\mathcal{D}C$, then both $B$ and $C$ are derived discrete by
Proposition~\ref{Proposition-DerivedDiscrete}. Since $\gl A =
\infty$, at least one of $B$ and $C$ is of infinite global dimension
\cite[Proposition 2.14]{AKLY13}. Without loss of generality, we
assume $\gl B = \infty$. Thus $B$ is derived equivalent to
$\tilde{B} = \Lambda(s',s',t') \oplus B'$ for some integers $s' \geq
1$ and $t' \geq 0$ and some algebra $B'$. Therefore, there is a
recollement
$$\xymatrix@!=4pc{ \mathcal{D}\tilde{B} \ar[r]^{i_*}
& \mathcal{D}A \ar@<-3ex>[l]_{i^*} \ar@<+3ex>[l]_{i^!} \ar[r]^{j^!}
& \mathcal{D}C \ar@<-3ex>[l]_{j_!} \ar@<+3ex>[l]_{j_*}}.$$ The
proposition will be proved by the following three steps.

\medskip

{\it Step 1.} We claim $s'=s$ and $t' \leq t$.

\medskip

For this, let $S'_0$ be the simple $\Lambda(s',s',t')$-module
corresponding to the vertex $0$ of the quiver $Q(s',t')$. Then
$S'_0$ is naturally a $\tilde{B}$-module as well and
$\Hom_{\mathcal{D}\tilde{B}}(S'_0,S'_0[h]) \cong k$ if
$h=0,s',2s',\cdots$, and $=0$ otherwise. Since $S'_0$ is a
noncompact indecomposable object in $\mathcal{D}^b(\tilde{B})$ and
$i^*$ preserves compactness, we have $i_*S'_0$ is a noncompact
indecomposable object in $\mathcal{D}^b(A)$. By
Lemma~\ref{Lemma-Hom-RHom}, we have
$\Hom_{\mathcal{D}A}(i_*S'_0,i_*S'_0[h]) \cong k$ if
$h=0,s,2s,\cdots$, and $=0$ otherwise. Therefore, $s'=s$, since the
triangle functor $i_*$ is fully faithful. It follows from
\cite[Proposition 6.5]{AKLY13} that the rank $r(\tilde{B})$ of the
Grothendieck group $K_0(\tilde{B})$ of $\tilde{B}$ is not larger
than the rank $r(A)$ of the Grothendieck group $K_0(A)$ of $A$.
Hence, $t' \leq t$.

\medskip

{\it Step 2.} We claim $\gl B' < \infty$.

\medskip

Assume on the contrary $\gl B' = \infty$. Then $B'$ is derived
discrete by Proposition~\ref{Proposition-DerivedDiscrete} and of
infinite global dimension. Thus $B'$ is derived equivalent to
$\Lambda(s'',s'',t'') \oplus B''$ for some integers $s'' \geq 1$ and
$t'' \geq 0$ and some algebra $B''$. Therefore, $i_*$ induces a
triangle functor
$$i'_* : \mathcal{D}(\Lambda(s',s',t')) \oplus
\mathcal{D}(\Lambda(s'',s'',t'')) \oplus \mathcal{D}B'' \simeq
\mathcal{D}\tilde{B} \rightarrow \mathcal{D}A.$$ Let $S''_0$ be the
simple $\Lambda(s'',s'',t'')$-module corresponding to the vertex $0$
of the quiver $Q(s'',t'')$. Then $\RHom_A(i'_*S'_0,i'_*S''_0) \cong
0$. However, $i'_*S'_0$ and $i'_*S''_0$ are pairwise non-isomorphic
noncompact indecomposable objects in $\mathcal{D}^b(A)$, and thus we
have $\RHom_A(i'_*S'_0,i'_*S''_0) \ncong 0$ by
Lemma~\ref{Lemma-Hom-RHom}. It is a contradiction.

\medskip

{\it Step 3.} We claim $\gl C < \infty$.

\medskip

Assume on the contrary $\gl C = \infty$. Then $C$ is derived
discrete by Proposition~\ref{Proposition-DerivedDiscrete} and of
infinite global dimension. Thus $C$ is derived equivalent to
$\tilde{C} = \Lambda(s''',s''',t''') \oplus C'$ for some integers
$s''' \geq 1$ and $t''' \geq 0$ and some algebra $C'$. Therefore,
there is a recollement
$$\xymatrix@!=4pc{ \mathcal{D}\tilde{B} \ar[r]^{i_*}
& \mathcal{D}A \ar@<-3ex>[l]_{i^*} \ar@<+3ex>[l]_{i^!} \ar[r]^{j^!}
& \mathcal{D}\tilde{C} \ar@<-3ex>[l]_{j_!} \ar@<+3ex>[l]_{j_*}}.$$
where both $i_*$ and $j_*$ can be restricted to
$\mathcal{D}^b(\mod)$ by Proposition~\ref{Proposition-n-recollement}
and \cite[Lemma 2]{HQ14}. Let $S'''_0$ be the simple
$\Lambda(s''',s''',t''')$-module corresponding to the vertex $0$ of
the quiver $Q(s''',t''')$. Then $\RHom _A(i_*S'_0,j_*S'''_0) \cong
0$. However, $i_*S'_0$ and $j_*S'''_0$ are noncompact indecomposable
objects in $\mathcal{D}^b(A)$, and thus $\RHom _A(i_*S'_0,j_*S'''_0)
\ncong 0$ by Lemma~\ref{Lemma-Hom-RHom} (2). It is a contradiction.
\end{proof}

As a consequence of Lemma~\ref{Lemma-Recollement-B-C}, we have the
following proposition which was obtained independently by Angeleri
H\"{u}gel, K\"{o}nig, Liu and Yang in different way in the earlier
version of \cite{AKLY13}.

\begin{proposition}\label{Proposition-2-TrunCycAlg}
Every 2-truncated cycle algebra is $n$-derived-simple for all $n
\geq 1$.
\end{proposition}

\begin{proof}
Assume $A = \Lambda (s,s,0)$ and $\mathcal{D}A$ admits an
$n$-recollement relative to $\mathcal{D}B$ and $\mathcal{D}C$.
According to Lemma~\ref{Lemma-Recollement-B-C}, we may assume that
$\gl C < \infty$ and $B$ is derived equivalent to $\Lambda(s,s,0)
\oplus B'$ for some algebra $B'$. By \cite[Proposition 6.5]{AKLY13},
we have $r(A) = r(B)+r(C) = r(A)+r(B')+r(C)$. Thus $r(B') = 0 =
r(C)$. Hence $B' = 0 = C$, i.e., the $n$-recollement is trivial.
\end{proof}

The next proposition implies that dropping an idecomposable
projective module with projective radical in \cite{IZ92} induces a
recollement of derived categories of algebras.

\begin{proposition} \label{Proposition-DropProjMod}
Let $A$ be a basic algebra, $e$ a primitive idempotent, $\bar{e} :=
1-e$, and $\rad eA$ a projective $A$-module. Then $\mathcal{D}A$
admits a recollement relative to $\mathcal{D}(\bar{e}A\bar{e})$ and
$\mathcal{D}k$.
\end{proposition}

\begin{proof}
Let $X$ be the projective $A$-module $\bar{e}A$, and $Y$ the simple
$A$-module $eA/\rad eA$. Since $\rad eA$ is projective, we have
$\pd_A Y \leq 1$. It is easy to see that both $X$ and $Y$ are
compact exceptional objects in $\mathcal{D}A$ with $\End_A(X) \cong
\bar{e}A\bar{e}$ and $\End_A(Y) \cong k$, and $\RHom_A(X,Y) \cong
0$. Moreover, for all $Z \in \mathcal{D}A$ satisfying $\RHom_A(X,Z)
\cong 0 \cong \RHom_A(Y,Z)$, we have $\RHom_A(X \oplus Y,Z) \cong
0$. Clearly, $\rad eA \in \thick X$, the smallest thick subcategory
of $\mathcal{D}A$ containing $X$, i.e., the smallest full
triangulated subcategory of $\mathcal{D}A$ containing $X$ and closed
under direct summands. Thus $eA$, and further $A \cong eA \oplus X$,
belongs to $\thick (X \oplus Y)$. Hence $Z \cong \RHom_A(A,Z) \cong
0$. Therefore, $X$ and $Y$ determine a recollement of $\mathcal{D}A$
relative to $\mathcal{D}k$ and $\mathcal{D}(\bar{e}A\bar{e})$ by
\cite[Proposition 1]{Han14}, and further a 2-recollement of
$\mathcal{D}A$ relative to $\mathcal{D}k$ and
$\mathcal{D}(\bar{e}A\bar{e})$ by \cite[Proposition 1]{HQ14}. Hence,
$\mathcal{D}A$ admits a recollement relative to
$\mathcal{D}(\bar{e}A\bar{e})$ and $\mathcal{D}k$.
\end{proof}

Let $B$ be a bound quiver algebra, $M$ a right $B$-module and $N$ a
left $B$-module. Then $\left[\begin{array}{cc} B & 0 \\ M & k
\end{array}\right]$ is called a {\it one-point extension} of $B$,
and $\left[\begin{array}{cc} k & 0 \\ N & B
\end{array}\right]$ is called a {\it one-point coextension} of $B$
(Ref. \cite{Rin84}). The following lemma implies that one-point
(co)extension induces a recollement, which will be used frequently
later.

\begin{lemma} \label{Lemma-One-Point-CoExtension} Let $A$ be a one-point
(co)extension of a bound quiver algebra $B$. Then $\mathcal{D}A$
admits a recollement relative to $\mathcal{D}B$ and $\mathcal{D}k$.
\end{lemma}

\begin{proof} Both one-point extensions and one-point coextensions are
special triangular matrix algebras. So this lemma follows from
\cite[Example 1 (2)]{HQ14}. In fact, the one-point coextension case
also can be deduced from Proposition~\ref{Proposition-DropProjMod}
by considering the idempotent element corresponding to the
coextension vertex which is a sink in the quiver of $A$ and
corresponds to a simple projective $A$-module.
\end{proof}

The following result is our main theorem in this section.

\begin{theorem} \label{Theorem-DerivedSimple}
For any positive integer $n$, a derived discrete algebra is
$n$-derived-simple if and only if it is derived equivalent to either
$k$ or a 2-truncated cycle algebra.
\end{theorem}

\begin{proof}
{\it Sufficiency.} It follows from
Proposition~\ref{Proposition-2-TrunCycAlg}.

{\it Necessity.} Let $A$ be a basic $n$-derived-simple derived
discrete algebra. Then $A$ must be connected. If $A$ is a piecewise
hereditary algebra of Dynkin type then it is derived equivalent to
$k$ by \cite[Theorem 5.7]{AKL12}. Otherwise, $A$ is derived
equivalent to $\Lambda(r,s,t)$ for some $s \geq r \geq 1$ and $t
\geq 0$. If $t \geq 1$ then $A$ is a one-point extension of
$\Lambda(r,s,t-1)$. It follows from
Lemma~\ref{Lemma-One-Point-CoExtension} and
Proposition~\ref{Proposition-n-recollement} that $\mathcal{D}A$
admits a non-trivial $n$-recollement relative to
$\mathcal{D}\Lambda(r,s,t-1)$ and $\mathcal{D}k$, which contradicts
to the assumption. Hence, we have $t = 0$. If $s > r$ then applying
Proposition~\ref{Proposition-DropProjMod} to the algebra
$\Lambda(r,s,0)$ and the idempotent $e$ corresponding to the vertex
$0$ of the quiver $Q(s,0)$ and also
Proposition~\ref{Proposition-n-recollement}, we can obtain a
non-trivial $n$-recollement of $\mathcal{D}\Lambda(r,s,0)$, and thus
$\mathcal{D}A$. It contradicts to the assumption, which implies $s =
r$ and $A$ must be derived equivalent to a $2$-truncated cycle
algebra $\Lambda(s,s,0)$.
\end{proof}

\section{\large Jordan-H\"{o}lder theorems}

\indent\indent In this section, we will show that the
Jordan-H\"{o}lder theorems hold for all kinds of derived categories
of derived discrete algebras.

\begin{definition}{\rm
Let $n$ be a positive integer. An {\it $n$-composition-series} of
the derived category $\mathcal{D}A$ of an algebra $A$ is a chain of
derived categories of algebras linked by fully faithful triangle
functors
$$0 = \mathcal{D}B_l \stackrel{i_l}{\hookrightarrow} \mathcal{D}B_{l-1}
\stackrel{i_{l-1}}{\hookrightarrow} \cdots
\stackrel{i_2}{\hookrightarrow} \mathcal{D}B_1
\stackrel{i_1}{\hookrightarrow} \mathcal{D}B_0 = \mathcal{D}A$$ such
that, for all $p = 1,2, \cdots, l$, the triangle functor
$\mathcal{D}B_p \stackrel{i_p}{\hookrightarrow} \mathcal{D}B_{p-1}$
can be completed to an $n$-recollement
$$\xymatrix@!=4pc{ \mathcal{D}B_p \ar[r]^-{i_p} & \mathcal{D}B_{p-1}
\ar@<-3ex>[l] \ar@<+3ex>[l]^-{\vdots} \ar[r] & \mathcal{D}C_p
\ar@<-3ex>[l] \ar@<+3ex>[l]^-{\vdots} }$$ for some
$n$-derived-simple algebra $C_p$. In this case, $\mathcal{D}C_1,
\mathcal{D}C_2, \cdots, \mathcal{D}C_l$ are called the {\it
$n$-composition-factors} of $\mathcal{D}A$. Moreover,
$\mathcal{D}C_p$ is said to be of {\it multiplicity} $m_p$ if it
appears exactly $m_p$ times in the sequence $\mathcal{D}C_1,
\mathcal{D}C_2, \cdots, \mathcal{D}C_l$ up to derived equivalence.
}\end{definition}

\begin{remark}{\rm (1) By \cite[Proposition 6.5]{AKLY13}, we know
that the length $l$ of an $n$-composition-series of $\mathcal{D}A$
is not larger than the rank $r(A)$ of the Grothendieck group
$K_0(A)$ of $A$.

(2) If $n \neq n'$ then it is possible that the length of an
$n$-composition-series of $\mathcal{D}A$ is not equal to the length
of an $n'$-composition-series of $\mathcal{D}A$. For example, let
$A$ be the two-point algebra given in \cite[Example 5.8]{AKLY13}
which is 2-derived simple but not 1-derived-simple. Then the length
of any $1$-composition-series of $\mathcal{D}A$ is 2. However, the
length of any $2$-composition-series of $\mathcal{D}A$ is 1.

(3) It is possible that the length of one $n$-composition-series of
$\mathcal{D}A$ is equal to the length of the other
$n$-composition-series of $\mathcal{D}A$ but two
$n$-composition-series have different $n$-composition-factors, so
the Jordan-H\"{o}lder theorems do not hold for derived categories of
algebras in general. For example, let $A$ be the two-point algebra
given in \cite[Example 7.6]{AKLY13}. Then $\mathcal{D}A$ admits two
$1$-composition-series with completely different
$1$-composition-factors.

(4) According to \cite[Theorem 5.7 and Corollary 5.9]{AKL12}, for
any piecewise hereditary algebra $A$ and $n =1,2,3$, $\mathcal{D}A$
admits an $n$-composition-series with only $n$-composition-factors
$\mathcal{D}k$ of multiplicity $r(A)$ which are independent of the
choice of $n$-composition-series up to derived equivalence. Thus the
Jordan-H\"{o}lder theorem holds for the unbounded (resp. upper
bounded, lower bounded, bounded) derived category of any piecewise
hereditary algebra. }\end{remark}

\begin{lemma} \label{Lemma-DirectSumRecollement}
Let $A_1,\cdots,A_m$ be algebras, $A = \oplus_{p=1}^m A_p$ and $n$ a
positive integer. If $\mathcal{D}A_p$ admits an
$n$-composition-series with $n$-composition-factors
$\mathcal{D}C_{p,1}, \cdots, \mathcal{D}C_{p,l_p}$ for all $1 \leq p
\leq m$ then $\mathcal{D}A$ admits an $n$-composition-series with
$n$-composition-factors $\mathcal{D}C_{1,1}, \cdots,
\mathcal{D}C_{1,l_1}, \cdots \cdots, \mathcal{D}C_{m,1}, \cdots,
\mathcal{D}C_{m,l_m}$.
\end{lemma}

\begin{proof}
It is clear that any $n$-recollement $$\xymatrix@!=3pc{ \mathcal{D}B
\ar[r] & \mathcal{D}A \ar@<-1ex>[l] \ar@<+1ex>[l]^-{\vdots} \ar[r] &
\mathcal{D}C \ar@<-1ex>[l] \ar@<+1ex>[l]^-{\vdots} }$$ can always
induce an $n$-recollement
$$\xymatrix@!=5pc{ \mathcal{D}(B \oplus E) \ar[r] & \mathcal{D}(A
\oplus E) \ar@<-1ex>[l] \ar@<+1ex>[l]^-{\vdots} \ar[r] &
\mathcal{D}C \ar@<-1ex>[l] \ar@<+1ex>[l]^-{\vdots} }$$ for all
algebras $A$, $B$, $C$ and $E$. Let
$$0 = \mathcal{D}B_{p,l_p} \hookrightarrow \cdots \hookrightarrow
\mathcal{D}B_{p,1} \hookrightarrow \mathcal{D}B_{p,0} =
\mathcal{D}A_p$$ be an $n$-composition-series of $\mathcal{D}A_p$
with $n$-composition-factors $\mathcal{D}C_{p,1}, \cdots, \linebreak
\mathcal{D}C_{p,l_p} = \mathcal{D}B_{p,l_p-1}$. Then they induce an
$n$-composition-series $0  = \mathcal{D}B_{m,l_m} \linebreak
\hookrightarrow \cdots \hookrightarrow \mathcal{D}B_{m,1}
\hookrightarrow \mathcal{D}B_{m,0} = \mathcal{D}A_m = \mathcal{D}A_m
\oplus \mathcal{D}B_{m-1,l_{m-1}} \hookrightarrow \mathcal{D}A_m
\oplus \mathcal{D}B_{m-1,l_{m-1}-1} \hookrightarrow \cdots
\hookrightarrow \mathcal{D}A_m \oplus \mathcal{D}B_{m-1,1}
\hookrightarrow \mathcal{D}A_m \oplus \mathcal{D}B_{m-1,0} =
\mathcal{D}A_m \oplus \mathcal{D}A_{m-1} \hookrightarrow \cdots
\cdots \hookrightarrow \mathcal{D}A$ of $\mathcal{D}A$ with
$n$-composition-factors $\mathcal{D}C_{1,1}, \cdots, \linebreak
\mathcal{D}C_{1,l_1}, \cdots \cdots, \mathcal{D}C_{m,1}, \cdots,
\mathcal{D}C_{m,l_m}$.
\end{proof}

The following result is our main theorem in this section.

\begin{theorem} \label{Theorem-JH}
Let $n$ be a positive integer and $A$ a derived discrete algebra,
say derived equivalent to $\oplus^m_{p=1}A_p$ where $m \geq 1$ and
$A_p = \Lambda(s_p,s_p,t_p)$ for $s_p \geq 1$, $t_p \geq 0$, $1 \leq
p \leq u$ and $0 \leq u \leq m$; $A_p = \Lambda(r_p,s_p,t_p)$ for
$s_p
> r_p \geq 1$, $t_p \geq 0$, $u+1 \leq p \leq v$ and $u \leq v \leq m$;
$A_p$ is a basic connected piecewise hereditary algebra of Dynkin
type for $v+1 \leq p \leq m$. Then $\mathcal{D}A$ admits an
$n$-composition-series with $n$-composition-factors
$\mathcal{D}(\Lambda (s_1,s_1,0)), \cdots , \mathcal{D}(\Lambda
(s_u,s_u,0))$, and $\mathcal{D}k$ of multiplicity $r(A)-\Sigma_{p=1}
^u s_p$. Moreover, any $n$-composition-series of $\mathcal{D}A$ has
precisely these $n$-composition-factors up to reordering and derived
equivalence.
\end{theorem}

\begin{proof}
For $v+1 \leq p \leq m$, $A_p$ is a basic connected piecewise
hereditary algebra of Dynkin type. By \cite[Theorem 1.1 (i)]{HRS96},
$A_p$ is {\it triangular}, i.e., its quiver has no oriented cycles.
Thus $A_p$ can be constructed from $k$ by $r(A_p)-1$ times of
one-point extensions. It follows from
Lemma~\ref{Lemma-One-Point-CoExtension},
Proposition~\ref{Proposition-DerivedDiscrete} and
Proposition~\ref{Proposition-n-recollement} that $\mathcal{D}A_p$
admits an $n$-composition-series with $n$-composition-factors
$\mathcal{D}k$ of multiplicity $r(A_p)$.

For $u+1 \leq p \leq v$, $A_p = \Lambda(r_p,s_p,t_p)$ with $s_p >
r_p \geq 1$ and $t_p \geq 0$. The algebra $\Lambda(r_p,s_p,t_p)$ can
be constructed from $\Lambda(r_p,s_p,0)$ by $t_p$ times one-point
extensions. For $\Lambda(r_p,s_p,0)$, since $s_p > r_p$, applying
Proposition~\ref{Proposition-DropProjMod} to the idempotent element
$e$ corresponding to the vertex $0$ of the quiver $Q(s_p,0)$ and
also Proposition~\ref{Proposition-n-recollement}, we get an
$n$-recollement of $\mathcal{D}\Lambda (r_p,s_p,0)$ relative to
$\mathcal{D}E_p$ and $\mathcal{D}k$ where $E_p$ is a Nakayama
algebra whose quiver is a line quiver. The algebra $E_p$ can be
constructed from $k$ by $r(E_p) - 1 = s_p - 2$ times of one-point
extensions. Therefore, $\mathcal{D}\Lambda(r_p,s_p,t_p)$ admits an
$n$-composition-series with $n$-composition-factors $\mathcal{D}k$
of multiplicity $r(\Lambda(r_p,s_p,t_p))=t_p+s_p$.

For $1 \leq p \leq u$, $A_p = \Lambda(s_p,s_p,t_p)$ with $s_p \geq
1$ and $t_p \geq 0$. The algebra $\Lambda(s_p,s_p,t_p)$ can be
constructed from $\Lambda(s_p,s_p,0)$ by $t_p$ times one-point
extensions. Thus $\mathcal{D}\Lambda (s_p,s_p,t_p)$ admits an
$n$-composition-series with $n$-composition-factors
$\mathcal{D}(\Lambda(s_p,s_p,0))$ and $\mathcal{D}k$ of multiplicity
$t_p$.

By the above analyzes and Lemma~\ref{Lemma-DirectSumRecollement}, we
know $\mathcal{D}A$ admits an $n$-composition-series with
$n$-composition-factors $\mathcal{D}(\Lambda (s_1,s_1,0)), \cdots ,
\mathcal{D}(\Lambda (s_u,s_u,0))$, and $\mathcal{D}k$ of
multiplicity $r(A) - \Sigma_{p=1} ^u s_p$.

Next we prove these $n$-composition-factors of $\mathcal{D}A$ are
independent of the choice of $n$-composition-series by induction on
$r(A)$. If $r(A) = 1$ then $A$ is local. Thus $A$ is
$n$-derived-simple. By Theorem~\ref{Theorem-DerivedSimple}, $A$ is
derived equivalent to either $k$ or $\Lambda(1,1,0)$. In either
case, we have nothing to say. Assume $r(A) \geq 2$ and the statement
holds for all algebras $B$ with $r(B) < r(A)$. For any
$n$-composition-series
$$(*) \;\;\;\;\;\;\; 0 = \mathcal{D}B_l \hookrightarrow \mathcal{D}B_{l-1}
\hookrightarrow \cdots \hookrightarrow \mathcal{D}B_1
\hookrightarrow \mathcal{D}B_0 = \mathcal{D}A $$ of $\mathcal{D}A$,
we assume that the triangle functor $\mathcal{D}B_1 \hookrightarrow
\mathcal{D}A$ is completed to an $n$-recollement
$$\xymatrix@!=4pc{ \mathcal{D}B_{1} \ar[r] & \mathcal{D}A
\ar@<-1ex>[l] \ar@<+1ex>[l]^-{\vdots} \ar[r] & \mathcal{D}C_1
\ar@<-1ex>[l] \ar@<+1ex>[l]^-{\vdots} }$$ where $C_1$ is
$n$-derived-simple. Then $C_1$ is derived discrete by
Proposition~\ref{Proposition-DerivedDiscrete}. Thus up to derived
equivalence $C_1 = k$ or $\Lambda(s,s,0)$ for some integer $s \geq
1$ by Theorem~\ref{Theorem-DerivedSimple}.

According to the proof of \cite[Corollary 3.4]{LY12}, we get $B_1 =
\oplus_{p=1}^m B_{1,p}$, $C_1 = \oplus_{p=1}^m C_{1,p}$ and the
above $n$-recollement can be decomposed as the direct sum of
$n$-recollements
$$\xymatrix@!=4pc{ \mathcal{D}B_{1,p} \ar[r] & \mathcal{D}A_p
\ar@<-1ex>[l] \ar@<+1ex>[l]^-{\vdots} \ar[r] & \mathcal{D}C_{1,p}.
\ar@<-1ex>[l] \ar@<+1ex>[l]^-{\vdots} }$$ The $n$-derived-simplicity
of $C_1$ implies that $C_1$ is connected. Therefore, there exists
some $q \in \{1,2, \cdots, m\}$ such that $C_1 = C_{1,q}$ and
$\mathcal{D}B_{1,p} \simeq \mathcal{D}A_p$ for all $p \neq q$. Up to
derived equivalence, we may assume $B_1 = (\oplus_{p \neq q} A_p)
\oplus B_{1,q}$.

If $1 \leq q \leq u$ then $A_q = \Lambda(s_q,s_q,t_q)$. By
Lemma~\ref{Lemma-Recollement-B-C}, we have either $C_1 = k$ and
$B_{1,q} = \Lambda (s_q,s_q,t'_q) \oplus B'_{1,q}$ for some integer
$0 \leq t'_q \leq t_q$ and some algebra $B'_{1,q}$ with $\gl
B'_{1,q} < \infty$, or $C_1 = \Lambda (s_q,s_q,0)$ and $\gl B_{1,q}
< \infty$. If $u+1 \leq q \leq m$ then $\gl A_q < \infty$. Thus $C_1
= k$ and $\gl B_{1,q} < \infty$. Let's discuss these three cases in
detail.

\medskip

{\it Case 1.} $A_q = \Lambda(s_q,s_q,t_q)$, $C_1 = k$ and $B_{1,q} =
\Lambda (s_q,s_q,t'_q) \oplus B'_{1,q}$ with $\gl B'_{1,q} <
\infty$.

\medskip

In this case $B_1 = (\oplus_{p \neq q}A_p) \oplus B_{1,q} =
(\oplus_{p \neq q}A_p) \oplus \Lambda(s_q,s_q,t'_q) \oplus
B'_{1,q}$, and $r(B_1) = r(A)-1$. By induction assumption, the
$n$-composition-series
$$0 = \mathcal{D}B_l \hookrightarrow \mathcal{D}B_{l-1}
\hookrightarrow \cdots \hookrightarrow \mathcal{D}B_2
\hookrightarrow \mathcal{D}B_1$$ of $\mathcal{D}B_1$ has exactly
$n$-composition-factors $\mathcal{D}(\Lambda (s_p,s_p,0))$ with $1
\leq p \leq u$ and $\mathcal{D}k$ of multiplicity $r(B_1) -
\Sigma_{p=1}^u s_p$ up to reordering and derived equivalence. Thus
the $n$-composition-series $(*)$ of $\mathcal{D}A$ has exactly
$n$-composition-factors $\mathcal{D}(\Lambda (s_p,s_p,0))$ with $1
\leq p \leq u$ and $\mathcal{D}k$ of multiplicity $r(B_1) -
\Sigma_{p=1}^u s_p + 1 = r(A) - \Sigma_{p=1}^us_p$  up to reordering
and derived equivalence.

\medskip

{\it Case 2.} $A_q = \Lambda(s_q,s_q,t_q)$, $C_1 = \Lambda
(s_q,s_q,0)$ and $\gl B_{1,q} < \infty$.

\medskip

In this case $B_1 = (\oplus_{p \neq q}A_p) \oplus B_{1,q}$ with $\gl
B_{1,q} < \infty$, and $r(B_1) = r(A)-s_q$. By induction assumption,
the $n$-composition-series
$$0 = \mathcal{D}B_l \hookrightarrow \mathcal{D}B_{l-1}
\hookrightarrow \cdots \hookrightarrow \mathcal{D}B_2
\hookrightarrow \mathcal{D}B_1$$ of $\mathcal{D}B_1$ has exactly
$n$-composition-factors $\mathcal{D}(\Lambda (s_p,s_p,0))$ with $1
\leq p \leq u$ and $p \neq q$, and $\mathcal{D}k$ of multiplicity
$r(B_1) - \Sigma_{p \neq q}s_p$, up to reordering and derived
equivalence. Thus the $n$-composition-series $(*)$ of $\mathcal{D}A$
has exactly $n$-composition-factors $\mathcal{D}(\Lambda
(s_p,s_p,0))$ with $1 \leq p \leq u$ and $\mathcal{D}k$ of
multiplicity $r(B_1) - \Sigma_{p \neq q} s_p = r(A) - \Sigma_{p=1}^u
s_p$ up to reordering and derived equivalence.

\medskip

{\it Case 3.} $\gl A_q < \infty$, $C_1 = k$ and $\gl B_{1,q} <
\infty$.

\medskip

In this case $B_1 = (\oplus_{p \neq q}A_p) \oplus B_{1,q}$ with $\gl
B_{1,q} < \infty$ and $r(B_1) = r(A)-1$. By induction assumption,
the $n$-composition-series
$$0 = \mathcal{D}B_l \hookrightarrow \mathcal{D}B_{l-1}
\hookrightarrow \cdots \hookrightarrow \mathcal{D}B_2
\hookrightarrow \mathcal{D}B_1$$ of $\mathcal{D}B_1$ has exactly
$n$-composition-factors $\mathcal{D}(\Lambda (s_p,s_p,0))$ with $1
\leq p \leq u$ and $\mathcal{D}k$ of multiplicity $r(B_1) -
\Sigma_{p=1}^u s_p$ up to reordering and derived equivalence. Thus
the $n$-composition-series $(*)$ of $\mathcal{D}A$ has exactly
$n$-composition-factors $\mathcal{D}(\Lambda (s_p,s_p,0))$ with $1
\leq p \leq u$ and $\mathcal{D}k$ of multiplicity $r(B_1) -
\Sigma_{p=1}^u s_p +1 = r(A) - \Sigma_{p=1}^u s_p$ up to reordering
and derived equivalence.

Now we finish the proof of the theorem.
\end{proof}

\begin{remark}{\rm For $n = 1$ (resp. 2, 3), Theorem~\ref{Theorem-JH}
implies that the Jordan-H\"{o}lder theorem holds for the unbounded
(resp. upper bounded or lower bounded, bounded) derived category of
any derived discrete algebra. }\end{remark}

\medskip

\noindent {\footnotesize {\bf ACKNOWLEDGMENT.}
The author is grateful to her supervisor Professor Yang Han for his guidance.
This work is supported by Project 11171325 NSFC.}

\end{document}